\documentclass[11pt,letterpaper]{amsart}
\usepackage{amssymb}
\usepackage{mathrsfs, tikz}
\usepackage[all,cmtip]{xy}
\usepackage{pstricks}

\usepackage{ulem}
\usepackage{comment}

\usepackage{graphpap,color}

\definecolor{cof}{RGB}{219,144,71}
\definecolor{pur}{RGB}{186,146,162}
\definecolor{greeo}{RGB}{91,173,69}
\definecolor{greet}{RGB}{52,111,72}

\numberwithin{equation}{section}

\newtheorem{thm}{Theorem}[section]
\newtheorem{defn}[thm]{Definition}
\newtheorem{prop}[thm]{Proposition}
\newtheorem{prop-defn}[thm]{Proposition-Definition}

\newtheorem{lemma}[thm]{Lemma}
\newtheorem{cor}[thm]{Corollary}

\newcommand{\Id}{{\rm Id}}

\newcommand{\CD}{\xymatrix@R=1pc@C=1pc}
\newcommand{\CDR}{\xymatrix@R=1pc}
\newcommand{\CDC}{\xymatrix@C=1pc}

\def\cB{{\mathcal B}}

\def\cD{{\mathcal D}}

\def\cP{{\mathcal P}}

\def\cZ{{\mathcal Z}}

\def\sF{{\mathscr F}}

\def\sR{\mathscr{R}}
\def\sV{\mathscr{V}}
\def\tsV{\widetilde{\mathscr{V}}}

\def\fG{\mathfrak{G}}

\def\fL{\mathfrak{L}}

\def\fU{\mathfrak{U}}
\def\fV{\mathfrak{V}}

\def\fr{\mathfrak{r}}

\def\fr{\mathfrak{r}}
\def\fs{\mathfrak{s}}

\def\NN{{\mathbb N}}
\def\PP{{\mathbb P}}

\def\ZZ{{\mathbb Z}}

\def\AA{{\mathbb A}}

\def\Ga{{\Gamma}}

\def\vt{{\varTheta}}

\def\vt{{\vartheta}}

\def\si{{\sigma}}

\def\var{{\rm Var}}

\def\lex{{\rm lex}}

\def\lra{\longrightarrow}

\def\kk{{\bf k}}

\def\bU{{\bf U}}
\def\bV{{\bf V}}

\def\bz{{\bf z}}

\def\bm{{\bf m}}
\def\bbm{\bar{\bf m}}

\def\barT{\bar{T}}

\def\uu{{\underbar {\it u}}}
\def\uv{{\underbar {\it v}}}

\def\pl{{\hbox{Pl\"ucker}}}

 \def\2{{\rm I\!I}}

\def\-{{\setminus}}

\def\ve{{\varepsilon}}

\def\vp{{\varpi}}
\def\vr{{\varrho}}

\def\vi{{\varphi}}

\def\cBkd1{{\cB_{[k]}^{d_\vr=1}}}
\def\cBd1{{\cB^{d_\vr=1}}}

\def\modcBkd1{{\mod \; \cB_{[k]}^{d_\vr=1}}}

\def\Jac{{\rm Jac}}

\def\rk{{\rm rank \;}}

\def\ngv{{\rm ngv}}

\def\ngvr{{\rm ngv1}}
\def\ngvh{{\rm ngv2}}

\def\gov{{\rm gov}}

\def\mh{{\hbox{\scriptsize ${\rm mh}$}}}

\def\de{\delta}

\def\tsR{{\widetilde{\sR}}}

\def\up{{\Upsilon}}

\makeatletter
\newcommand*\bigcdot{\mathpalette\bigcdot@{.7}}
\newcommand*\bigcdot@[2]{\mathbin{\vcenter{\hbox{\scalebox{#2}{$\m@th#1\bullet$}}}}}
\makeatother

\makeatletter
\def\@tocline#1#2#3#4#5#6#7{\relax
  \ifnum #1>\c@tocdepth 
  \else
    \par \addpenalty\@secpenalty\addvspace{#2}%
    \begingroup \hyphenpenalty\@M
    \@ifempty{#4}{%
      \@tempdima\csname r@tocindent\number#1\endcsname\relax
    }{%
      \@tempdima#4\relax
    }%
    \parindent\z@ \leftskip#3\relax \advance\leftskip\@tempdima\relax
    \rightskip\@pnumwidth plus4em \parfillskip-\@pnumwidth
    #5\leavevmode\hskip-\@tempdima
      \ifcase #1
       \or\or \hskip 1em \or \hskip 2em \else \hskip 3em \fi%
      #6\nobreak\relax
    \hfill\hbox to\@pnumwidth{\@tocpagenum{#7}}\par
    \nobreak
    \endgroup
  \fi}
\makeatother

\begin{document}

\title[Fundamental Lemma]
{Fundamental Lemmas in Universal Resolution}

\date{}
\author{Yi Hu}

\maketitle

\begin{abstract} 
We abstract and prove some fundamental lemmas in universal blowups.
 \end{abstract}

\maketitle



\section{Introduction}
\medskip

The results of this paper stand on their own.

Moreover, they serve to validate \cite{Hu2025}.

As a special case of the present work, 
Chapters 5, 6 and 8 of \cite{Hu2025} follow automatically.

\section{The standardized models}

For any natural number  $n$, we let $[n]=\{1, \cdots, n\}$.

\subsection{The standardized models} $\ $

Let $\bU$ be the affine space 
$\AA^N$ with coordinates $(x_i)_{i \in [N]}$.

\begin{defn}\label{stand-models}
Consider the closed subset $\bV \subset \bU$ defined by
\begin{equation}\label{pseu-pl}
F_k = x_k + \sum_{s=1}^{\tau_k} \bm_{(k,s)}, \;  k \in  [\up]
\end{equation}
for some natural numbers $\tau_k$ and $\up$,
where $\bm_{(k,s)}$ are monomials, such that
\begin{enumerate}
\item {\rm (existence of leading variable)}
$x_k \nmid \bm_{(k,s)}$ for any $s \in [\tau_k]$;

There exists $0<\Xi <\up$ such that 

\item {\rm (pleasantness)} if  $x_k \mid \bm_{(k',t)}$
 for some $k, k' \in[\up]$ and $t \in [\tau_{k'}]$, then  
$k\in [\Xi], k' >\Xi$ and
$x_j \nmid \bm_{(k',t)}$ for all $j \ne k$.
\end{enumerate}
We call the above a standardized model of the closed subset $V$.
\end{defn}
Clearly, $\bV$ is isomorphic to the affine space $\AA^{N-\up}$ with coordinates $(x_i)_{i > \up}$. 
Let $\Ga \subset \{x_i\}_{i \in [N]}$ and $Z_\Ga \subset
\bV$ be defined by $x_i =0, i \in [N]$. Then, the strata
$Z_\Ga$ can admit arbitrary singularities, up to smooth
morphisms, according to the lemmas below.

\begin{prop}\label{Gr-is-Smodel}
A chart of a Grassmannian 
with the primary $\pl$ relations\footnote{
For a quick reference, see  Theorem 1.4 with 
$\Ga=\emptyset$, \cite{Hu2026}; for details,
see Chapter 3, \cite{Hu2025}.} 
is a  standardized model.
\end{prop}
\begin{proof} This can be checked easily from definition 
(see Chapter 3, \cite{Hu2025}).
\end{proof}

\begin{prop}\label{La-Mnev}
{\rm (Lafforgue's version
of Mn\"ev's universality. cf. Theorem 1.2 of \cite{Hu2026}.)} For any affine variety $X$ defined over 
$\ZZ$, there exists $\Ga$, an open susbet $Z^\circ_\Ga$ of $Z_\Ga$, and a smooth morphism
$Z^\circ_\Ga \lra X$
\end{prop}

\subsection{Standard birational transforms} $\ $

For any $k \in  [\up]$, we let $\bm_{k,0}=x_k$, 
 we can express
$$F_k = \sum_{s=0}^{\tau_k} \bm_{(k,s)}.$$ 
We introduce the projective space
$$\PP_{F_k}, \; \hbox{equipped with the homogenous coordinates}\; [x_{(k,s)}]_{0 \le s \le \tau_k}.$$
We  then introduce the rational map 
$$\Theta: \bV  \lra  \prod_{k=1}^\up \PP_{F_k}$$
$$(x_i)_i \to  \prod_{k=1}^\up [\bm_{(k,s)}]_{0 \le s \le \tau_k},$$
 let $\sV$ be the closure of the graph of $\Theta$, and obtain the diagram
$$ \xymatrix{
\sV \ar[d] \ar @{^{(}->}[r]   \ar[d] \ar @{^{(}->}[r]  &
\sR:= \bU  \times \prod_{k=1}^\up \PP_{F_k} \ar[d] \\
\bV  \ar @{^{(}->}[r]    & \bU.}
$$
The left vertical morphism is birational.\footnote{Standard birational transform generalizes to familiies of relations,
e.g.,   using Gr\"obner bases. }

\begin{lemma}\label{universal-eqs}
The scheme $\sV$ as a closed subset of $\sR$ is defined by 
\begin{itemize}
\item{\rm GL} {\rm  (Governing Linearized Relations)}
$$L_{F_k}: \sum_{s =0}^{\tau_k}  x_{(k, s)}, 
\; \forall \; k \in [\up];$$
\item{\rm GB} {\rm (Governing Binomials)}
\\[-2em]
\begin{eqnarray}\label{GovB}\nonumber
B_{(k,s)}: 
x_kx_{(k,s)} - \bm_{(k,s)} x_{(k,0)}, \; s \in [\tau_k], \; k \in [\up]
\end{eqnarray}
\\[-3em]
\item{\rm NGB1}   {\rm (Non-governing Binomials, I)}
\\[-2em]
\begin{eqnarray}\label{nGovB}\nonumber
B_{k; (s,t)}: 
\bm_{(k,t)} x_{(k,s)} - \bm_{(k,s)} x_{(k,t)}, \; s \ne t \in [\tau_k], \; k \in [\up].
\end{eqnarray}
\\[-3em]
\item{\rm NGB2} {\rm (Non-governing Binomials, II)} 
This  subset  consists of all relations  of $\sV$ in $\sR$ that are
algebraically independent of relations in {\rm (GL), (GB)}, and {\rm (NGB1)}.
\end{itemize}
\end{lemma}
\begin{proof}  The birational map $\Theta$ corresponds to the following homomorphism
\begin{equation}\label{vi}
\vi: R_0 \otimes \bigotimes_{F \in \sF} \kk[x_{(k,s)}]_{0 \le s \le \tau_k}
 \lra R_0/\langle F_k \mid k \in [\up] \rangle
\end{equation}
$$ \vi|_{R_0}=\Id_{R_0}, \;  x_{(k,s)} \to \bm_{(k,s)}$$
where $R_0=\kk[x_j]_j$ is the affine algebra of $\bU$,
 and $\kk[x_{(k,s)}]_{0 \le s \le \tau_k}$ is graded. Then, $\sV \subset \sR$ is defined by
the multi-homogeneous kernel $\ker^\mh (\vi)$. 

One checks directly that relations in (GL), (GB), and (NGB1) belong to $\ker^\mh (\vi)$.
\end{proof}

We let 
\begin{itemize}
\item $\sF=\{F_1, \cdots, F_\up\}$,
 $L_\sF=\{L_F \mid F \in \sF\}$;
\item $\cB^\gov$ be the set of all 
governing binomials in (GB);
\item 
 $\cB^\ngvr$ be the set of all binomials in (NGB1), and
\item $\cB^\ngvh$ be the set of all binomials in  (NGB2). 
(We can require $\cB^\ngvh$ to be finite.)
\item We set $\cB^\ngv=\cB^\ngvr \sqcup \cB^\ngvh$.
\end{itemize}

For any $F_k \in \sF$, we let $\cB^\gov_{F_k}$ or
just $ \cB^\gov_k$ be  consist of all the governing binomials associated with $F_k$ (those as written  in Lemma \ref{universal-eqs} (GB))
and  $$\fG_{F_k}=\{L_{F_k}\} \sqcup \cB^\gov_{F_k}.$$
Then we obtain  the set of  all governing relations
$$\fG=\{\fG_{F_1} < \cdots < \fG_{F_\up}\}.$$ 

Note that any binomial $B_{k; (s,t)}$ of $\cB^\ngvr$ 
is associated with $F_k$, 
and we let $\cB^\ngvr_{F_k}$ be  consist of all the non-governing binomials associated with $F_k$ 

\begin{defn}
The smooth scheme $\sR$ comes equipped with two kinds of variables
\begin{itemize}
\item we call $x_i$ a $\vp$-variable;
\item  we call $x_{(k,s)}$ a $\vr$-variable.
\end{itemize}
\end{defn}

\begin{prop-defn}\label{pleas} For any $F_k \in \sF$, the  variable $x_k$ and the $\vr$-variables of 
$\cB^\gov_{F_k}$ do not appear in any relations of $\fG_{F_j}$ 
where $j <k$.
Such a variable is called pleasant.
\end{prop-defn}

\section{Universal blowups}\label{uni-blowups}

\subsection{Outlines of the universal blowups}

\begin{defn} For any $j \in [N]$, we let 
$$X_j=(x_j =0) \subset \sR, $$
called a $\vp$-divisor. For any coordinate $x_{(k,s)}$ of $\PP_{F_k}$ for some $k \in [\up]$,
we let 
$$X_{(k,s)}=(x_{(k,s)}=0) \subset \sR,$$
called $\vr$-divisor. For any $F \in \sF$, we let 
$$D_F=(L_F =0) \subset \sR,$$  called a $\fL$-divisor. 
\end{defn}

We let $\cD_{\sR,\vp}$ be the set of all $\vp$-divisors, 
$\cD_{\sR, \vr}$ the set of all $\vr$-divisors, and 
$$\cD_\sR=\cD_{\sR, \vp} \sqcup \cD_{\sR,\vr}.$$ 

\begin{defn}\label{order-0}
We provide a total order on the set $\cD_\sR$ as follows.
\begin{itemize}
\item Suppose both $i, j \in [\up]$ or both $i, j > \up$. Then, $X_i < X_j$ if $i < j$;
\item $X_i < X_j$ if $i > [\up]$ and $ j \in [\up]$;
\item $X_{(k',s')} <X_{(k,s)}$ if $(k',s') <_\lex (k,s)$;
\item $X_j < X_{(k,s)}$ for all $j$ and all $(k,s)$.
\end{itemize}
Here $<_\lex$ denotes the lexicographic order.
\end{defn}

We inductively construct our universal blowups, 
block by block, according to the order
$$\fG_{F_1}< \fG_{F_2}< \cdots < \fG_{F_\up}.$$

\begin{defn}\label{defn-divisors}
For any universal blowup $\tsR_* \lra \tsR_{*-1}$, inductively we let
$$\cD_{\sR_*} = \widehat\cD_{\sR_{*-1}} \sqcup E_*$$
where $\widehat \cD_{\sR_{*-1}}$ consists of proper transforms of divisors of $ \cD_{\sR_{*-1}} $
and $E_*$ is the exceptional divisor of the blowup $\tsR_* \lra \tsR_{*-1}$.
\end{defn}

\begin{defn}\label{defn-divisor-order}
 We inductively introduce a total order on the set $\cD_{\sR_*}$.
The base case $\cD_\sR$ is ordered in Definition \ref{order-0}. 
Assume that $\cD_{\sR_{*-1}}$  is totally ordered.
Then  the set $\cD_{\sR_*}$ is totally ordered by 
\begin{itemize}
\item we let $\widehat\cD_{\sR_{*-1}}$ inherit the order of $\cD_{\sR_{*-1}}$;
\item $E_* < D$ for any $D \in  \widehat\cD_{\sR_{*-1}}$.
\end{itemize}
\end{defn}

\begin{defn}\label{defn-divisor-order2}
 We let $\cD_{\sR_*} \times \cD_{\sR_*}$ be ordered lexicographically, using the order
of $\cD_{\sR_*}$.
\end{defn}

For any governing binomial
$B=x_k x_{(k,s)} -\bm_{(k,s)}x_{(k,0)}$ 
with $k \in [\up]$ and $s \in [\tau_k]$, we let
\begin{equation}\label{TBpm}
T_B^+ = x_k x_{(k,s)}\;\; \hbox{and} \;\;
 T_B^-=\bm_{(k,s)}x_{(k,0)}, 
\end{equation}
called plus-term and minus-term of $B$, respectively.
The  center of  a  $\vt$- or  $\wp$-blowup $\sR_{*} \lra \sR_{*-1}$ will be constructed from
  the proper transform of 
$T_B^\pm$ for some  $B\in \cB^\gov$.

Inductively, we let $\tsV_* \subset \tsR_*$
be the proper transform of  $\tsV_{*-1}$ 
($\subset \tsR_{*-1}$), with
the base case being $\sV \subset \sR$.
We cover $\tsV_{*-1}$ 
by smooth charts $\{\fV'\}$ of $\tsR_{*-1}$.

\begin{defn}\label{defn-p/t}
For any chart $\fV'$, we write the proper transforms of
$$T_B^\pm, \; L_F, \; x_j, \; x_{(j, s)}, \; X_{(j,s)}, 
\; \hbox{etc.}, $$
over  the chart $\fV'$, as
  $$T_{\fV', B}^\pm, \; L_{\fV',F}, \; x_{\fV', j}, \; x_{\fV', (j,s)}, \; X_{\fV',(j,s)},  \; \hbox{ ect.}.$$ 
\end{defn}

\begin{defn} 
For any chart $\fV'$ of $\sR_{*-1}$, 
we consider  all possible 
factor $y_\pm$ in  the proper transform $T_{\fV', B}^\pm$ over $\fV'$.
Each of $(y_\pm=0)$  corresponds to  global divisor $D^\pm \in \cD_{\sR_{*-1}}$.
Then,  a universal $\vt$- or $\wp$-blowup 
$$\pi: \tsR_{*-1} \lra \sR_*$$ is the blowup of $\tsR_{*-1}$ along $D^+ \cap D^-$.
In particular, locally over $\fV'$, 
$$\pi^{-1}(\fV') \lra \fV'$$ is  the blowup of $\fV'$ along $(y_+=0)\cap (y_-=0).$
\end{defn}
Universal $\vt$-blowups are 
defined in \S \ref{vt-blowups} 
and universal $\wp$-blowups are defined 
in \S \ref{wp-blowups}.  
Universal $\ell$-blowups are
in different flavour and are introduced in \S \ref{ell-blowups}.

\subsection{Introducing and on $\vt$-blowups} \label{vt-blowups} $\ $

Fix any  $k \in [\up]$ and consider all $B=T_B^+-T_B^- \in \cB^\gov_{F_k}$. We have
$$T_B^+ = x_k x_{(k,s)}, \;  T_B^-=\bm_{(k,s)}x_{(k,0)}, \;\;
s \in [\tau_k].$$
Observe that the pair $(x_k, x_{(k,0)})$ appears in every member of $\cB^\gov_{F_k}$.
Thereby we  let 
$$Z_{\vt_{F_k}}  = X_k \cap X_{(k,0)}.$$
Then we have
$$Z_{\vt_{F_1}} < Z_{\vt_{F_2}} < \cdots < Z_{\vt_{F_\up}},$$
and we blow up $\sR$ along these centers, in that order,  to obtain
$$\tsR_\vt:=\tsR_{\vt_{F_\up}} \lra \cdots \lra \tsR_{\vt_{F_1}} \lra \sR.$$
It induces
$$\tsV_\vt:=\tsV_{\vt_{F_\up}} \lra \cdots \lra \tsV_{\vt_{F_1}} \lra \sV$$
where  $\tsV_{\vt_{F_k}}$ is the proper transform of $\tsV$. 

Each blowup $\tsR_{\vt_{F_j}} \lra \tsR_{\vt_{F_{j-1}}}$, where $\tsR_{\vt_{F_0}} =\sR$,
 creates an exceptional divisor $E_{\vt_{F_j}}$ 
of  $\tsR_{\vt_{F_j}}$,
and we let $E_{\vt_k, F_j}$ be the proper transform of  $E_{\vt_{F_j}}$ in $\tsR_{\vt_k}$.

\begin{lemma}\label{vtk-blowup}  
The scheme $\tsR_{\vt_k}$ is smooth for every $k \in [\up]$. We have  
 \begin{equation}\label{vtk-adhoc}
\tsV_{\vt_k} \cap X_{\vt, (k,0)} = \emptyset, \;\;
\hbox{ for all $k \in [\up]$} .
\end{equation}
We can cover 
$\tsV_{\vt_k}$  by 
smooth charts $\{\fV\}$ of  $\tsR_{\vt_k}$
such that for any  $\fV$, the following hold.
\begin{enumerate}
\item For  any  binomial $B=B_{(j,s)}$ in $\cB^\gov_{F_j}$ with $j \in [\up]$ and $s \in [\tau_j]$, we have
$$T_{\fV, (j,s)}^+ =x_{\fV, j} x_{\fV, (j, s)}.$$
\item The proper transforms of all relations in $B_{i; (s,t)} \in \cB^\ngvr_{F_i}$ with $i \in [k]$
become dependent and can be discarded from consideration.
\item For any $j \in [\up]$, when $j>k$, or,
 when $j \in [k]$ and $E_{\vt_k, F_j} \cap \fV =\emptyset$,  
$$(\alpha) \;\;\;\; 
L_{\fV, F_j} = x_{\fV, (j,0)} + \sum_{s =1}^{\tau_j} x_{\fV, (j,s)};$$
when  $j \in [k]$ and $E_{\vt, F_j}\cap \fV \ne \emptyset$,
$$(\beta) \;\;\;\; 
L_{\fV, F_j} = \de_{\fV,(j,0)} + \sum_{s =1}^{\tau_j} x_{\fV, (j,s)}$$
where $\de_{\fV, (j,0)}$  is a coordinate of $\fV$ such that
$E_{\vt, F_j}\cap \fV =(\de_{\fV, (j,0)}=0)$.
\end{enumerate}
\end{lemma}
\begin{proof} 
We prove by induction on $k \in \{0\} \sqcup [\up]$. 
The  base case $k=0$ is for $\tsR:=\tsR_{\vt_0}$, in which case,  \eqref{vtk-adhoc} and (2) are void,
all other statements hold.

Fix $k \in [\up]$, we assume that all statements hold for  $``<k-1"$.

We let $\pi_k: \tsR_{\vt_k} \lra \tsR_{\vt_{k-1}}$
 be the blowup of $\tsR_{\vt_{k-1}}$ along the proper transform of
$Z_{\vt_{F_k}}  = X_k \cap X_{(k,0)}$, where  $\tsR_{\vt_0}=\tsR$.
Then, we have  $\tsR_{\vt_k} \subset \tsR_{\vt_{k-1}} \times \PP^1_{[\xi,\eta]}$.
We can assume $\tsR_{\vt_{k-1}}$ is covered by
 affine charts  $\{\fV'\}$  with a system  $\var_{\fV'}$ of affine coordinates for every chart $\fV'$.
Then over any fixed chart $\fV'$, the blowup center is
$$(x_{\fV', k}=0) \cap (x_{\fV', (k,0)}=0),$$
and $\pi_k^{-1}(\fV') \subset \fV'  \times \PP^1_{[\xi,\eta]}$ is defined by
$$\eta x_{\fV', k} - \xi x_{\fV', (k,0)}.$$

Clearly, the scheme $\tsR_{\vt_k}$ is smooth  because locally over any chart $\fV'$, 
it is the blowup of $\fV'$ along  the smooth center $x_{\fV',k} \cap x_{\fV', (k,0)}$.

The open subset $\pi_k^{-1}(\fV')$ is covered by two charts $(\xi \ne 0)$ and $(\eta \ne 0)$.

Over $(\xi \ne 0)$, we have
$$\eta x_{\fV', k} - x_{\fV', (k,0)}$$
where $x_{\fV', k}$ becomes the exceptional variable defining $E_{\vt_k} \cap \fV'$ and $\eta$
becomes  $x_{\fV, (k,0)}$, the proper transform of  $x_{\fV', (k,0)}$.

Over $(\eta \ne 0)$, we have
\begin{equation} \label{eta}
x_{\fV', k} - \xi x_{\fV', (k,0)}
\end{equation}
where $x_{\fV', (k,0)}$ becomes the exceptional variable defining $E_{\vt_k} \cap \fV$ and $\xi$
becomes  $x_{\fV, k}$, the proper transform of  $x_{\fV', k}$.

We claim that $(\eta \ne 0)$ covers $\pi_k^{-1}(\fV') \cap \tsV_{\vt_k}$. First, if $\fV'$ lies over 
the chart $(x_{ (k,0)} \ne 0)$, thus $\fV'$ is disjoint with the blowup center, hence the claim holds.
Next, suppose $\fV'$ lies over the chart $(x_{ (k,s_0)} \ne 0)$ for some $s_0 \in [\tau_k]$.
Then over the chart $(x_{ (k,s_0)} \ne 0)$, we have
$$B_{(k, s_0)}: x_k - m_{(k,s_0)} x_{(k,0)}.$$   Over the chart $(\xi \ne 0)$,  we can substitute 
$x_{\fV', (k,0)}=\eta x_{\fV', k}$ into (the proper transform of) the above displayed relation and obtain
$$B_{\fV, (k, s_0)}: 1 - m_{(k,s_0)} \eta,$$
thereby $\eta \ne 0$ on the chart, implying the claim.

In the statements of the lemma and in what follows, we cover $\tsV_{\vt_k}$ by the  charts $\{\fV\}$ of 
$\tsR_{\vt_k}$, either 
lying over $(x_{(k,0)} \ne 0)$ or defined by $(\eta \ne 0)$.

Now observe that over $(\eta \ne 0)$,
$x_{\fV', (k,0)}$ becomes the exceptional variable $\de_{\fV, (k,0)}$,
and also the blowups in $\tsR_{\vt_i}$ with $i \in [k]$
do not affect variables $x_{(s,t)}$ for $s, t \in [\tau_j]$ for any $j$
and  those in $L_{F_j}$ with $j>k$. This implies (1) as well as both (3) ($\alpha$) and (3) ($\beta$).

It remains to prove (2). 
Consider any $$B_{k; (s,t)}:  \bm_{(k,t)} x_{(k,s)} - \bm_{(k,s)} x_{(k,t)}, \; s \ne t \in [\tau_k], \; k \in [\up].$$
Over the chart $\fV$, lying over $(x_{(k,0)} \ne 0)$ or defined by $(\eta \ne 0)$, we have
$$B_{\fV, k; (s,t)}:  \bm_{\fV, (k,t)} x_{\fV,(k,s)} - \bm_{\fV, (k,s)} x_{\fV, (k,t)}, \; s \ne t \in [\tau_k], \; k \in [\up].$$
We also have $$B_{\fV, (k,s)} :  x_{\fV, k} x_{\fV, (k,s)}- \bm_{\fV, (k,s)}, \; 
B_{\fV, (k,t)} :  x_{\fV, k} x_{\fV, (k,t)}- \bm_{\fV, (k,t)}.$$
Then, one calculates and finds
$$B_{\fV, k; (s,t)}=  x_{\fV, (k,t)}B_{(k,s)} - x_{\fV, (k,s)}B_{(k,t)}.$$
Hence, (2) follows by induction.
\end{proof}

By the end of the sequential $\vt$-blowups, we obtain
$\tsV_\vt \subset \tsR_\vt$.

We let $E_{\vt, F_k}$ be the proper transform of  $E_{\vt_{F_k}}$ in $\tsR_{\vt}$.


\begin{cor}\label{vt-blowup}  
The scheme $\tsR_\vt$ is smooth.
We have $$\tsV_\vt \cap X_{\vt, (k,0)} = \emptyset, \;\;
\hbox{ for all $k \in [\up]$. } $$ Consequently, $\tsV_\vt$  can be covered by 
affine charts $\{\fV\}$ such that
\begin{enumerate}
\item For  any  binomial $B=B_{(k,s)}$ in $\cB^\gov_{F_k}$ for some $k \in [\up]$ and $s \in [\tau_k]$, we have
$T_{\fV, B}^+ =x_{\fV, k} x_{\fV, (k, s)}$,
keeping its original form.
\item The proper transforms of all relations in $\cB^\ngvr$
become dependent and can be discarded from consideration.
\item For any $k \in [\up]$, when  $E_{\vt, F_k} \cap \fV =\emptyset$  
$$(\alpha) \;\;\;\; 
L_{\fV, F_k} = x_{\fV, (k,0)} + \sum_{s =1}^{\tau_k} x_{\fV, (k,s)};$$
when $E_{\vt, F_k}\cap \fV \ne \emptyset$,
$$(\beta) \;\;\;\; 
L_{\fV, F_k} = \de_{\fV,(k,0)} + \sum_{s =1}^{\tau_k} x_{\fV, (k,s)}$$
where $\de_{\fV, (k,0)}$  is a coordinate of $\fV$ such that
$E_{\vt, F_k}\cap \fV =(\de_{\fV, (k,0)}=0)$.
\end{enumerate}
\end{cor}
\begin{proof}
This follows immediately from Lemma \ref{vtk-blowup}.
\end{proof}

\subsection{Introducing $\wp$-blowups}\label{wp-blowups} $\ $

After obtaining $\tsR_\vt$,  
we continue to inductively construct our universal blowups, block by block, according to the order  
$\fG_{F_1}< \fG_{F_2}< \cdots < \fG_{F_\up}.$

As mentioned earlier, inductively, we let $\tsR_* \lra \tsR_{*-1}$ be an (intermediate)  blowup
and $\tsV_* \subset \tsR_*$ be the proper transform of $\tsV_{*-1} \subset \tsR_{*-1}$.

\begin{defn}\label{termi}
Let $\tsR_*$ be any (intermediate) blouwp scheme 
such that $\tsV_*$ is covered by smooth charts $\{\fV\}$ of $\tsR_*$. We say
$B \in \cB^\gov$ terminates on $\fV$ if the proper transform 
$T_{\fV, B}^\pm$ is invertible along $\fV \cap \tsV_*$. 
We say $B$ terminates on $\tsR_*$ if it terminates on all charts $\fV$.
\end{defn}

The goal is to make every $B \in \cB^\gov$ terminate, block by block.

\smallskip

Inductively, suppose we have constructed all the universal blowups with respect to 
the block $\fG_{k-1}$, starting from $\tsR_\vt$.
We write the last blowup scheme constructed as $\tsR_{\ell_{k-1}}$ with the 
set $\cD_{\tsR_{\ell_{k-1}}}$ of divisors (Definition \ref{defn-divisors}).
We then move on to consider all the universal blowups,
namely, all the $\wp$-blowups and the $\ell$-blowup with respect 
to $\fG_k$.

First, we introduce  $\wp$-blowups  with respect 
to $\fG_k$, this process is denoted by $(\wp_k)$.  

The intitial case $(\wp_0)$ is $\tsV_\vt \subset \tsR_\vt$.
The process  $(\wp_k)$ may contain a multiple rounds, each round is denoted
by  $(\wp_k\fr_\mu)$ for some natural number $\mu$, and in each round, 
we construct a finite sequential blowups, each of which is named as a step,
and is denoted by $(\wp_k\fr_\mu\fs_h)$ for some natural number $h$.

We construct $(\wp_k\fr_\mu)$ inductively on $\mu$.

\begin{defn}\label{ktaumuh} Suppose $(\wp_k\fr_{\mu-1})$ has been contructed for 
some natural number $\mu$, and we obtain $\tsV_{(\wp_k\fr_{\mu-1})} \subset
\tsR_{(\wp_k\fr_{\mu-1})}$ such that  $\tsV_{(\wp_k\fr_{\mu-1})}$ is covered by
smooth charts $\{\fV'\}$ of  $\tsR_{(\wp_k\fr_{\mu-1})}$. Further,  $\tsR_{(\wp_k\fr_{\mu-1})}$
comes equipped with  the set  $\cD_{(\wp_k\fr_{\mu-1})}$ of
$\vp$-, $\vr$-, and exceptional divisors, as well as the set $\fL_{(\wp_k\fr_{\mu-1})}$ 
of the proper transforms of all $\fL$-divisors. 

Consider every governing binomial $B \in \fG_k$. 
Locally over any chart $\fV'$, we select all possible variables $y_\pm \mid T_{\fV',B}^\pm$ from  the monomial $T_{\fV',B}^\pm$ 
with the corresponding global divisors $D^\pm \in \cD_{(\wp_k\fr_{\mu-1})}$.
We keep the pair 
$(D^+,D^-)$ if $$D^+ \cap D^- \cap \tsV_{(\wp_k\fr_{\mu-1})} \ne 
\emptyset,$$ and in this case, call $(y_+, y_-)$ and also $(D^+,D^-)$, a $\wp_k$-set,
and $D^+ \cap D^-$ a $\wp_k$-center.
We denote the set of $\wp$-centers with respect to $\fG_k$ by
$$\cZ_{(\wp_k \fr_\mu)} = \{D^+ \cap D^- \mid D^+ \cap D^- \cap \tsV_{(\wp_k\fr_{\mu-1})} \ne  \emptyset\}.$$
We embed $\cZ_{(\wp_k \fr_\mu)}$ into 
$ \cD_{(\wp_k\fr_{\mu-1})} \times  \cD_{(\wp_k\fr_{\mu-1})}$ by sending
$(D^+ \cap D^-)$ to $(D^+, D^-)$. We let 
$\cZ_{(\wp_k \fr_\mu)}$  inherit
the lexicographic order 
from that of  $ \cD_{(\wp_k\fr_{\mu-1})} \times  \cD_{(\wp_k\fr_{\mu-1})}$
 (Definition \ref{defn-divisor-order2}).
We then blow up along the centers in $\cZ_{(\wp_k \fr_\mu)} $, in that given order.
This completes one round of $\wp$-blowups,
which we denote $(\fr_\mu)$ for some 
$\mu \in \NN$.
If the center $D^+ \cap D^-$ is the $h$-th member in $\cZ_{(\wp_k \fr_\mu)} $,
we then denote the $\wp$-blowup along $D^+ \cap D^-$  as
$$(\wp_k\fr_\mu\fs_h),$$
which reads: phase $k$, round $\mu$, and step $h$.
\end{defn}

Clearly each $\cZ_{(\wp_k \fr_\mu)}$  is finite, thus,
we have $h \in [\si_{k\mu}]$ where $\si_{k\mu}$ is the cardinality of 
$\cZ_{(\wp_k \fr_\mu)}$. We will prove that all $B \in \fG_k$ terminate after a finitely many
rounds of $\wp$-blowups, that is, we have $ \mu \in  [\rho_k]$ for some positive integer $\rho_k$
depending on $k$.

\subsection{Introducing $\ell$-blowups}\label{ell-blowups} $\ $

When the governing binomials of $\fG_k$ all terminate (which
will be justified in Lemma \ref{wp:key} (4)), 
we obtain the final $\wp$-blowup  schemes for that block
$$\tsV_{\wp_k} \subset \tsR_{\wp_k},$$
equipped with the proper transform 
$E_{\wp_k, F_k}$ of $E_{\vt, F_k}$ 
(cf. Corollary \ref{vt-blowup}) and the proper transform $D_{\wp_k, F_k}$ of $D_{F_k}=(L_{F_k}=0)$.
Now, we blow up $ \tsR_{\wp_k}$ along
$$E_{\wp_k, F_k} \cap D_{\wp_k, F_k}$$ to obtain the $\ell$-blowup: 
$$ \tsR_{\ell_k} \lra \tsR_{\wp_k}.$$ 

\subsection{Notaitonal setups} $\ $

The notational setup here will be 
applied in the proof of Lemma \ref{wp:key} (4).

Consider any intermediate universal $\wp$-blowup
$$\tsR_* \lra \tsR_{*-1}.$$  
Let $\tsV_{*} \subset \tsR_{*-1}$ 
be the proper transform of $\tsV_{*-1}$. We cover 
$\tsV_{*}$ by smooth charts $\{\fV\}$ of $\tsR_{*}$,
lying over smooth charts $\{\fV'\}$ of 
$\tsR_{*-1}$. For any chart $\fV$, we let $\var_\fV$ be the set of variables on the chart, inductively, 
consisting of the proper transforms of the variables of 
$\var_{\fV'}$ and the  exceptional variable $e$ of the blowup $\tsR_* \lra \tsR_{*-1}$ such that
$$(e=0) = E_* \cap \fV,$$
where $E_*$ is
the exceptional divisor of $\tsR_* \lra \tsR_{*-1}$.

\begin{defn}\label{setup}
Locally over a chart $\fV'$ of $\sR_{*-1}$, the blowup is  along 
$$(y_+=0) \cap (y_-=0)$$ where $y_\pm \mid T_{\fV',B}^\pm$ for 
some $B=B_{(k,s)} \in  \cB^\gov_{F_k}$. 
The preimage $\pi^{-1}(\fV')$ as a closed subset of $\fV' \times \PP^1_{[\xi,\eta]}$  is
defined by $$\xi y_- - \eta y_+.$$
We have
$\pi^{-1}(\fV') \subset \fV' \times \PP^1_{[\xi,\eta]}$  is covered by 
 $(\xi \ne 0)$ and $(\eta \ne 0)$, thus, $\fV$ corresponds to 
either $(\xi \ne 0)$ or $(\eta \ne 0)$.
On the chart $\fV$ corresponding to
 $(\xi \ne 0)$, we write $y_- =\eta y_+$. In this case,
$y_+=e$ is the exceptional parameter such that $(e=0) = E_* \cap \fV$,
$\eta$ is the proper transform of $y_-$, and recycle the symbol to use
$y_-$ for $\eta$.
On the chart $\fV$ corresponding to
 $(\eta \ne 0)$, we write $y_+ =\xi y_-$. In this case,
$y_-=e$ is the exceptional parameter, and
$\xi$ is the proper transform of $y_-$, and recycle the symbol to use
$y_+$ for $\xi$. 
\end{defn}

\begin{defn}\label{subs} 
By Definition \ref{setup}, 
we can write the $\wp$-blowup center as
$$(x=0)\cap (y=0) $$
for two variables of $\fV'$. W.l.o.g., to compute the proper transform of any polynomial, we can substitute $x$ by $e$, the exceptional parameter, and substitute $y$ by $ye$. In symbols, 
$$ x \leadsto e, \;\; y \leadsto ye.$$
\end{defn}

Using Definition \ref{subs}, 
we suppose  a governing binomial
$B_{\fV'}$ can be written  as
$$B_{\fV'}: \; x f - yg$$
(without specifiying $\pm$-terms and without specifying the multiplicities of $x$ and $y$), then
after the substitution 
$$ x \leadsto e, \;\; y \leadsto ye,$$ we obtain
the proper transform
\begin{equation}\label{either-chart}
B_\fV: \; f- y g.
\end{equation}


\begin{defn}\label{position}
Let $p=(x, B, \si)$ with 
$x \in \var_{\fV'}, B \in \cB^\gov_{F_k}, \si \in \{+, -\}$, called a position.
We let  $\cP_{\fV'}$ be the set of all such positions. This is a finite set.
For any $p=(x, B, \si) \in \cP_{\fV'}$,
if $x$ is not invertible, we let the $\wp$-weight or just weight for short, $w_p$ of $p$, 
be the multiplicity of $x$ in $T^\si_{\fV', B}$;
 if $x$ is invertible,
 we let $w_p=0$. This way, we obtain a weight vector  $$(w_p)_{p \in \cP_{\fV'}}.$$
\end{defn}

\subsection{On $\wp$-blowups} $\ $

As in Definition \ref{ktaumuh}, 
we let  $\tsR_{\wp_k\fr_{\mu-1}}$
 be the scheme obtained after performing all the
$\wp$-blowups in $(\wp_k\fr_{\mu-1})$; we let $\tsV_{\wp_k\fr_{\mu-1}}$
be the proper transform of $\sV$.

We cover $\tsV_{\wp_k\fr_{\mu-1}}$ by smooth charts
$\{\fV'\}$ of $\tsR_{\wp_k\fr_{\mu-1}}$.

\begin{defn}\label{mulitiplicity-free}
For any $B \in \cB^\gov_k$, 
we say  the proper transform 
$B_{\fV'}$ is $\wp$-multiplicity-free 
or just multiplicity-free if for any variable
$y \mid T^\pm_{\fV',B}$, either $y$ is invertible or
the multiplicity of $y$ equals 1.
We say $B$ is 
multiplicity-free on $ \tsR_{\wp_k\fr_{\mu-1}}$,
if it is multiplicity-free on all charts $\{\fV'\}$.
\end{defn}

\begin{defn}
A $\wp_k$-center $D^+ \cap D^-$ is called tamed if 
for any given chart $\fV'$ having variables
$y_\pm$ on $\fV'$ such that $(y_+=0)\cap (y_-=0)= D^+\cap D^- \cap \fV'$, 
whenever $y_\pm \mid T^\pm_{B,\fV'}$,
we have $y_\mp \mid T^\mp_{B,\fV'}$, for all
$B \in \cB^\gov_k$.
\end{defn}

\begin{lemma}\label{termi} 
Let the notation be as in above.
 Assume that for any chart $\fV'$,
$B_{\fV'}$ is  $\wp$-multiplicity-free for 
all $B \in \cB_k^\gov$, 
 and all $\wp_k$-centers are tamed.
Then, every $B \in \cB_k^\gov$ terminates 
after all the $\wp$-blowups are performed in this round,
that is, every $B \in \cB_k^\gov$ terminates in $(\wp_k\fr_{\mu})$.
\end{lemma}
\begin{proof}
We let $\tsV_{\wp_k\fr_{\mu}} \subset \tsR_{\wp_k\fr_{\mu}} $ obtained after performing all the
$\wp$-blowups in $(\wp_k\fr_{\mu})$ are performed.
We cover $\tsV_{\wp_k\fr_{\mu}}$ 
by smooth charts  $\{\fV\}$ of
$\tsR_{\wp_k\fr_{\mu}} $. 
Suppose $B_\fV=T^+_{\fV, B} - T^-_{\fV, B}$ 
does not terminate  for some $B \in \cB^\gov_k$ and
chart $\fV$. Then, there are
$y_\pm \mid T^\pm_{\fV, B}$ such that
\begin{equation}\label{no-such-center}
(y_+=0)\cap (y_-=0) \cap \fV \ne \emptyset.
\end{equation}
Then, because $B_\fV$ is $\wp$-multiplicity-free, 
and all centers are tamed, 
using \eqref{either-chart},
we see that $y_\pm$ must be proper transform of some 
$y'_\pm \mid 
T^\pm_{\fV', B}$ for some $\fV'$ where $\fV$ lies over.  But, 
the blowup along the center $$(y'_+=0) \cap (y'_-=0) \cap \fV' $$
has been performed, then again, 
by \eqref{either-chart},
it is impossible to have
\eqref{no-such-center}. 
\end{proof}

In the lemma below, we use the setup in Definition \ref{setup}.
So we have $\pi: \tsR_* \lra \tsR_{*-1}$, a chart $\fV$
of  $\tsR_*$ lying over a chart $\fV'$ of  $ \tsR_{*-1}$,
and $\pi^{-1}(\fV') \lra \fV'$ is the blowup along 
$(y_+=y_-=0)$ with $y_\pm \mid T^\pm_{\fV', B}$ for some
$B=B_{(k,s)} \in \cB^\gov$.

\begin{lemma}\label{wp:key}  
Consider any chart $\fV$ of $\tsR_*$ 
lying  over $\fV'$ of $\sR_{*-1}$. 
We can write $B_{\fV'}$ as
$$ y_+\barT^+_{\fV', (k,s)} -
y_-^m \barT^-_{\fV', (k,s)}$$
where $m$ is the multiplicity  of $y$ for some integer $m>0$.
\begin{enumerate}
\item On $(\xi \ne 0)$, $y_-=\eta y_+$, we have  
$$B_{\fV, (k,s)}: \; \barT^+_{\fV', (k,s)}  -
  e^{m-1} y_-^m \barT^-_{\fV, (k,s)}.$$
\item 
On $(\eta \ne 0)$, $y_+ = \xi y_-$, we have 
$$B_{\fV, (k,s)}: \; y_+\barT^+_{\fV', (k,s)}  -
  e^{m-1}  \barT^-_{\fV, (k,s)}.$$
\item $T^+_{\fV, B}$ is square-free for all $B \in \fG$.
\item Every $B \in \fG_k$ terminates
after finitely many $\wp$-blowups with respect to $\fG_k$. 
\end{enumerate}
\end{lemma}
\begin{proof} 
We prove the lemma by induction on $(\wp_k\fr_\mu)$. We assume
that all the statement hold for $(\wp_k\fr_{\mu-1})$.
By  induction,
 the multiplicity of $y_+$ in $T^+_{\fV', (k,s)}$ equals 1.

(1) and (2) follow immediately by the substitutions
$x \leadsto e, y \leadsto ye$ per Definition \ref{subs}.

(3). 

If $B \in \fG_j$ with $j <k$, 
the square-freeness of $T^+_{\fV, B}$ follows by (4) and induction.

If $B=B_{(j,a)} \in \fG_j$ with $j > k$ and $a \in [\tau_j]$, then 
$$T^+_{\fV, B}=x_{\fV, j} x_{(\fV, (j,a)}$$ by the pleasantness of Definition
\ref{stand-models}.

It remains to consider $B \in \fG_k$.

The square-freeness of $T^+_{\fV, (k,s)}$ follows from (1) and (2).

Consider any $t \ne s$. 
Then only possibility of $T^+_{\fV, (k,t)}$ becoming non-square-free occurs  when  we have
\begin{equation}\label{imp-1}
 y_+\barT^+_{\fV', (k,s)} - y_-^m \barT^-_{\fV', (k,s)},
\end{equation}
$$ y_- y_+\barT^+_{\fV', (k,t)} -
T^-_{\fV', (k,t)}.$$
But, this implies that $y_-<y_+$,
hence by the order of $\wp$-blowup,
the center $(y_-, *)$ with some $* \mid T^-_{\fV', (k,t)}$ precedes $(y_+, y_-)$, making the above impossible.
Thus, $T^+_{\fV, (k,t)}$ remains to be square-free.

Before proving (4), we need the following.
{\it Claim: The following pattern is impossible.}
\begin{equation}\label{imp-2}
 y_+\barT^+_{\fV', (k,s)} - y_-^m \barT^-_{\fV', (k,s)},
\end{equation}
$$ T^+_{\fV', (k,t)} -
y_+ y_- \barT^-_{\fV', (k,t)}.$$
Suppose this occurs, then to get $y_+$ in 
$y_+ y_- \barT^-_{\fV', (k,t)}$, there must be 
$y  \mid T^-_{\fV'', (k,t)}$ (which could be $y_-$) for
some earlier chart $\fV''$ that $\fV'$ lies over and
a  binomial $B$ in $\cB^\gov_k$ distinct from $B_{(k,s)}$ and $B_{(k,t)}$ such that 
\begin{equation}\label{imp-3}
B_{\fV''}:  y_+ f - y g,\end{equation} 
$$ T^+_{\fV'', (k,t)} -
y_+ y \barT^-_{\fV'', (k,t)},$$
and 
$(y_+, y)$ is the center on the chart $\fV''$.
But, by a simple induction, \eqref{imp-3} is impossible.

(4).

Inductively, we let  $\tsR_{\ell_{k-1}}$ be the schemes obtained by
all the universal $\wp$- and $\ell$-blowups with respect to
$\fG_1< \cdots < \fG_{k-1}$, starting from $\tsR_{\ell_{-1}}:=\tsR_\vt$.

In this part, we argue from $\tsR_{\ell_{k-1}}$ and on.
We cover $\tsV_{\ell_{k-1}}$ by smooth charts $\{\fU \}$ of 
$\tsR_{\ell_{k-1}}$. By Definition  \ref{position},
we obtain
$$N^0_{\fU}=(w_p)_{p \in \cP_{\fU}}, \; \forall \; \fU.$$

Starting from $\tsR_{\ell_{k-1}}$, we let the consecutive sequential 
$\wp$-blowups be written as
$$\tsR_{\ell_{k-1}} \longleftarrow \tsR_1 \longleftarrow \tsR_2 
\longleftarrow \cdots.$$
We cover $\tsV_j$ (the proper transform of $\tsV$) by smooth charts
$\{\fV_j\}$ of $\tsR_j$. Consider any chain of charts
$$\fU \longleftarrow \fV_1 \longleftarrow \fV_2 \longleftarrow \cdots$$
such that $\fV_j$ lies over $\fV_{j-1}$, while $\fV_0=\fU$.

We have the weight vector $N^0_{\fU}=(w_p)_{p \in \cP_{\fU}}$, inductively, we assume 
 the weight vector 
$N^{j-1}_{\fU}=(w^{j-1}_p)_{p \in \cP_{\fU}}$ is defined.

Consider $\fV_j$ over $\fV_{j-1}$.
For any $p=(x,B, \si) \in \cP_\fU$, we have
$w^{j-1}_p$. Inductively, using Definition \ref{subs}, any variable
of $x$ of $\var_\fU$ is transformed to a unique variable $x_{j-1}$ of $\var_{\fV_{j-1}}$. We let $e_j$ be the exceptional variable of $\fV_j$ for the blowup $\tsR_j \lra \tsR_{j-1}$. There are three possibilities.
\begin{enumerate}
\item if $x_{j-1} \leadsto e_j$, we let $w^j_{p}$ be the weight of $e_j$;
\item if $x_{j-1} \leadsto x_{j-1} e_j $, we let $w^j_{p}$ be the weight of $e_j$;
\item in all the remaining cases, we let $w^j_{p}=w^{j-1}_{p}$.
\end{enumerate}
This way, we obtain a new weight vector 
$N^j_{\fU}=(w^j_p)_{p \in \cP_{\fU}}$. Then by (1) and (2) of
this lemma, and by the impossibilities of the patterns
\eqref{imp-1} and \eqref{imp-2}, we obtain
\begin{equation}\label{ex-wt}
N^0_{\fU}>_{\rm cw} N^1_{\fU} >_{\rm cw} N^2_{\fU} >_{\rm cw} \cdots \end{equation}
where $(a_p)_{p \in \cP_{\fU}}>_{\rm cw} (b_p)_{p \in \cP_{\fU}}$
if $a_p \ge b_p$ for all $p \in \cP_{\fU}$, 
and at least one of the inequalities is strict.
Therefore,  the descending chain of weight vectors
\eqref{ex-wt} has to terminate to become stable, as
the smallest weight is zero. 
This implies that after finite steps of $\wp$-blowups, 
it has to cease to create new exceptional variables for any chart.
That is,  at this step, the following two hold:
(1) any variable is either invertible, 
or else, has multiplicity 0 or 1;
(2) any $\wp$-center is tamed.
 For if there is $x^2$ appears in $B_\fV$ for some $B \in \cB^\gov_k$
on some chart $\fV$ and $x$ is not invertible, then it
will certainly create new exceptional variable $e$ by (1) and (2) of
this lemma, a contradiction.
Likewise, if any center $(y_+, y_-)$ is not tamed, it 
will also certainly create new exceptional variable $e$.

Thus, after a finite step of $\wp$-blowups, the proper transforms
of $B \in \cB^\gov_k$ must all be $\wp$-multiplicity-free, and all $\wp$-centers must be tamed,
to cease to produce
any exceptional variables on any chart.

Then, from here, by applying Lemma \ref{termi}, we conclude that all $B \in \fG_k$ terminate.
\end{proof}

Below, we continue to follow
the setup in Definition \ref{setup}.

\begin{lemma}\label{termi-charts} 
Suppose $B_{\fV', (k,s)}: y_+f  - y_- g$,
and after blowing up along $(y_+=y_-=0)$,
$B_{\fV, (k,s)}$ terminates, then
$\tsV_* \cap \fV$ is covered by $(\xi \ne 0) \cap (\eta \ne 0)$.
\end{lemma}
\begin{proof} 
We have 
$\pi^{-1}(\fV') \subset  \fV' \times \PP^1_{[\xi, \eta]}$
is defined by $ y_+ \eta - y_- \xi =0.$

Over $(\xi \ne 0)$, we have
$$B_{\fV, (k,s)}: f - \eta g.$$
Since, it terminates, obtain $\eta \ne 0$ along $\tsV_* \cap \fV$.

Similarly, over $(\eta \ne 0)$, we have
$$B_\fV: \xi f -  g.$$
Since, it terminates, obtain $\xi \ne 0$ along $\tsV_* \cap \fV$.

This implies $\tsV_* \cap \fV$ is covered by $(\xi \ne 0) \cap (\eta \ne 0)$.
\end{proof}

\subsection{On $\ell$-blowups} $\ $

Now we consider the $\ell$-blowup
$$ \tsR_{\ell_k} \lra  \tsR_{\wp_k}$$ with the induced
projection
$$ \tsV_{\ell_k} \lra  \tsV_{\wp_k}.$$

\begin{lemma}\label{LF-forms}  
The scheme $\tsV_{\ell_k}$ can be covered by smooth charts of $\{\fV\}$
$\tsR_{\ell_k}$. Over any chart $\fV$ of $\tsR_*$ 
lying  over $\fV'$ of  $\tsR_{\wp_k}$,  we have either
\begin{equation}\label{case-alpha}\nonumber
(\alpha) \;\; 
L_{\fV, F_k} =  x_{\fV, (0,k)}+ \sum_{i=1}^l  x_{\fV, (k,s_i)} + 
 \sum_{j=1}^q x^*_{\fV, (k,t_j)}
\end{equation}
where $x^*_{\fV, (k,t_j)}$ 
denotes the pull-back, and $l+q=\tau_k$;  or
\begin{equation}\label{case-beta}\nonumber
(\beta) \;\;\;\;\;\; 
L_{\fV, F_k} =   y_{\fV, (k,0)} +1
\end{equation}
where $y_{\fV, (k,0)}$ is the proper
transform of the local parameter 
$\de_{(k,0)}$ for $E_{\vt,F_k}$.
\end{lemma}
\begin{proof} 
We let $\pi: \tsR_{\ell_k} \lra \sR$ be the induced projection. Take any point $\bz \in \tsV_{\ell_k}$
and let $\bz_0=\pi(\bz) \in \sR$.

Suppose $x_{ (k,0)} (\bz_0) \ne 0$. Then, we 
can assume $\bz \in \fV$, a chart of $\tsR_{\ell_k}$,
lying over a chart $\fV'$ of $\tsR_{\wp_k}$, 
and all lying over the chart $(x_{ (k,0)} \ne 0)$ of $\sR$.
Then, $x_{\fV', (k,0)}$  is invertible and never been used in any blowup. Thus, it is clear we get
$$L_{\fV, F_k} =  x_{\fV, (0,k)}+ 
\sum_{i=1}^l  x_{\fV, (k,s_i)} + 
 \sum_{j=1}^q x^*_{\fV, (k,t_j)}$$
where $x_{(k,s_i)} (\bz_0) \ne 0, 1 \le i \le l$
and $x_{(k,t_j)} (\bz_0) = 0, 1 \le j \le q$
for some $l, q$ with $l+q=\tau_k$.

Next, suppose $x_{ (k,0)} (\bz_0) =0$.
We let $E_{\wp_k, F_k}$ be the proper transform
of $E_{\vt, F_k}$ (cf. Corollary \ref{vt-blowup}).
 Then after all $\vt$-blowups, 
 all $\wp$- and $\ell$-blowups 
with respect to $\fG_j$ for all $j < k$,
and all $\wp$-blowups 
with respect to $\fG_k$,
we obtain 
$$L_{\fV', F_k} =\de_{\fV', (k,0)}+ 
\sum_{i=1}^l x_{\fV', (k, s_i)} + 
 \sum_{j=1}^q (t_j) x^*_{\fV', (k, t_j)}$$
where $(\de_{\fV', (k,0)}=0) = 
E_{\wp_k, F_k}\cap \fV'$,
 $x_{(k,s_i)} (\bz_0) \ne 0, 1 \le i \le l$
and $x_{(k,t_j)} (\bz_0) = 0, 1 \le j \le q$
for some $l, q$ with $l+q=\tau_k$.

Now, we blow up $ \tsR_{\wp_k}$ along
the $\ell_k$-center $ E_{\wp_k, F_k} \cap 
D_{\wp_k, F_k}$
 to obtain the $\ell$-blowup: 
$$ \tsR_{\ell_k} \lra \tsR_{\wp_k}.$$ 
The $\ell_k$-center 
$ E_{\wp_k, F_k} \cap  D_{\wp_k, F_k}$ is smooth because locally over any chart $\fV'$,
$$E_{\wp_k, F_k} \cap  D_{\wp_k, F_k} \cap \fV'=
(\de_{\fV', (k,0)}=0) \cap (L_{\fV', F_k}=0) \cap  \fV'$$ is
defined by 
$$\de_{\fV', (k,0)}=0, \;\;\; L^*_{\fV, F_k} =  \sum_{i=1}^l  x_{\fV', (k,s_i)} + 
 \sum_{j=1}^q x^*_{\fV', (k, t_j)},$$
and $l \ge 2$ so that at least one of $\{x_{\fV', (k,s_i)}\}$
provides a linear variable in $L^*_{\fV, F_k}$.
Hence, the $\ell_k$-center is smooth along $\tsV_{\wp_k}$
so that we can cover 
$\tsV_{\ell_k}$ by smooth charts $\{\fV\}$ of
 $\tsR_{\ell_k}$.

The closed subset
$\pi^{-1}(\fV') \subset \fV' \times \PP^1_{[\xi, \eta]}$
is defined by
$$\eta \de_{\fV', (k,0)} - \xi L^*_{\fV', F_k}.$$
Note that
$$L_{\fV', F_k} =  \de_{\fV', (k,0)}  + L^*_{\fV', F_k}$$
Over $(\xi \ne 0)$, we have 
$L^*_{\fV', F}= \eta \de_{\fV', (k,0)} $.
Hence 
$$L_{\fV, F_k} =1+  \eta  .$$
Thus $\eta$ is invertible along $\fV \cap \tsV_{\ell_k}$.
Over $(\eta \ne 0)$, we have 
$ \de_{\fV', (k,0)} =\xi  \ L^*_{\fV', F_k} $.
Hence 
$$L_{\fV, F_k} = \xi  + 1.$$
Thus $\xi$ is invertible along $\fV \cap \tsV_{\ell_k}$.
This implies that $\pi^{-1}(\fV')$ is covered by
$(\eta \ne 0) \cap (\xi \ne 0)$.
Therefore, we can choose to use the chart $(\eta \ne 0)$,
over which 
$$L_{\fV, F_k} = y_{\fV, (k,0)}  + 1.$$
where $y_{\fV, (k,0)}=\xi$ 
is the proper transform of $ \de_{\fV', (k,0)}$.

Thus,  by the above discussion,
the case when $x_{ (k,0)} (\bz_0) \ne 0$ gives  
 ($\alpha$),  while the case when $x_{ (k,0)} (\bz_0) = 0$
gives ($\beta$).
\end{proof}

\subsection{On the final blowup schemes $\tsV_\ell \subset \tsR_\ell$}
\label{the-final-scheme}  $\ $

We let $\tsR_\ell$ be the final blowup scheme and $\tsV_\ell$ the proper transform of $\sV$.

\begin{thm}\label{final-forms} 
 The  blowup scheme
$\tsV_\ell$  can be covered by  smooth  charts $\{\fV\}$ of $\tsR_\ell$
such that the following holds. 
\begin{enumerate}
\item For  any governing binomial $B_{(k,s)}$ in $\cB^\gov$, we have either
$$ T_{\fV, (k,s)}^+=   a x_{\fV, (k,s)}\;\; \hbox{or} \;\; 
T_{\fV,(k,s)}^+= a y_{\fV, k} x_{\fV, (k,s)}$$
where $a$ invertible, 
$y_{\fV,k}$ is a transform of  $x_k$.
Furthermore, $y_{\fV, k}$ and $x_{\fV, (k,s)}$  are
 both pleasant variables.
\item For any $F_k \in \sF$,  $L_{\fV, F_k}$ takes one of the following two forms:
\begin{equation}\label{case-alpha}\nonumber
(\alpha) \;\; 
L_{\fV, F_k} =  x_{\fV, (0,k)}+ \sum_{i=1}^l  x_{\fV, (k,s_i)} + 
 \sum_{j=1}^q x^*_{\fV, (k,j)}
\end{equation}
with $l+q=\tau_k$,
where $x^*_{\fV, (\uu_{s_j},\uv_{s_j})}$ denotes the pull-back; 
\begin{equation}\label{case-beta}\nonumber
(\beta) \;\;\;\;\;\; 
L_{\fV, F} =  y_{\fV, (k,0)} +1
\end{equation}
where $y_{\fV, (k,0)}$ is the proper
transform of the local parameter $\de_{(k,0)}$ for $E_{\vt, F_k}$.
\end{enumerate}
\end{thm}
\begin{proof}
We let $\pi: \tsR_{\ell} \lra \sR$ be the induced projection. 

Take any point $\bz \in \tsV_{\ell}$
and let $\bz_0=\pi(\bz) \in \sR$.

(1).

Suppose $x_{ (k,s)} (\bz_0) \ne 0$. 
 We know that {\it there
must be $s_0 \in [\tau_k]$ such that 
$x_{ (k,s_0)} (\bz_0) \ne 0$.}
Then, by the order of $\wp$-blowups, 
the largest non-invertible variable in the plus term $T_B^+$
is $x_k$. Hence, by applying Lemma \ref{termi-charts},
over a chart $\fV$ containing $\bz$, we should obtain
 $T_{\fV, (k,s)}^+= a  y_{\fV, k} x_{\fV, (k,s)}$ 
with $a$ being invertible. 

Suppose $x_{ (k,s)} (\bz_0) = 0$. Then,
by the order of $\wp$-blowups, and again, by applying 
Lemma \ref{termi-charts}, when $B_{(k,s)}$ terminates
at the chart $\fV'$ that $\bz$ lies over, 
$x_{\fV', (k,s)}$ becomes an invertible variable. This implies,
over a chart $\fV$ containing $\bz$, we have
 $T_{\fV, (k,s)}^+=   a  x_{\fV, (k,s)}$ with $a$ being invertible.

It is clear that both $y_{\fV, k}$ and $ x_{\fV, (k,s)}$ 
are pleasant.

(2) follows directly from 
Lemma \ref{LF-forms}.
\end{proof}

\section{Smoothness by Jacobians}\label{jjj}

We follow the noation of Theorem \ref{final-forms} and
 its proof.

As in the proof of Theorem \ref{final-forms},
we fix any chart $\fV$ and consider any point $\bz \in \fV$ lying over a point $\bz_0 \in \sV$.
 We consider $B_{\fV, (k,s)}$ for all $B_{(k,s)} \in \cB^\gov_k$.

\subsection{Jacobian matrix: Case ($\alpha$)} $\ $

In this case, $x_{(k,0)}(\bz_0) \ne 0$,
we can assume that $\fV$ lies over $(x_{(k,0)} \ne 0)$,
thereby $ x_{\fV, (k,0)}=1$.
By Theorem \ref{final-forms} (1) and (2) ($\alpha$), we 
obtain
\begin{eqnarray}\label{finaleq-alpha}  
   B_{\fV, (k,s_i)}: 
\;\; a_i x_{\fV,(k, s_i)} y_{\fV,k} - c_i, \;\; i \in [l]    \\
  B_{\fV, (k,t_i)}: \;\;  b_i x_{\fV,(k, t_i)} - d_i , \;\; i \in [q]. 
\;\;\;\;\;\; \nonumber  \\
\;\;\;\;\;\;\;\;\;\;\;\;\;\;\;
 L_{\fV, F} =1 + \sum_{i=1}^l   x_{\fV, (k, s_i)}
  + \sum_{i=1}^q  e_i  x_{\fV, (k, t_i)} 
\nonumber 
\end{eqnarray}
for some invertible monomials $a_i, b_i, c_i, d_i$, and $e_i$,
where $e_i$ are monomials in exceptional variables such that $e_i(\bz)=0$, for all $i \in [q]$.

 {\it As the chart $\fV$ is fixed and is clear from the context, in the sequel, to save space,
we will selectively drop some subindex $``\ \fV \ "$. For instance, we may write
$y_{k}$ for $y_{\fV,k}$, $x_{(k, s_i)}$ 
for $x_{\fV, (k, s_i)}$, etc., a confusion is unlikely.
}

\smallskip

We  introduce the following maximal minor of the Jacobian 
  $J(\fG_{\fV,F_k})$
$$J^*(\fG_{\fV,F_k})= {{\partial(B_{\fV, s_1} \cdots B_{\fV, s_{\l}}, B_{\fV, t_1} \cdots B_{\fV, t_q},
L_{\fV,F_k})} \over {{\partial(y_{k},
x_{(k,s_1)} \cdots x_{(k,s_l)},
x_{(k, t_1)} \cdots x_{(k, t_q)}
 )}}} .$$
 
 \begin{lemma}\label{jjj-alpha}
The matrix $J^*(\fG_{\fV,F_k})$ is invertible
at the point $\bz$, and all the variables that are used to compute it are pleasant.
\end{lemma}
 \begin{proof}
 One calculates and finds that at the point $\bz$,
$J^*(\cB^\gov_{\fV,F_k}, L_{\fV,F_k})$ is equal to
\begin{eqnarray} \nonumber
{\footnotesize
\left(
\begin{array}{cccccccccc}
a_1x_{(k,s_1)} & a_1y_{k}   & \cdots & 0 & 0  & \cdots &0 \\
\vdots \\
a_l x_{(k,l)} & 0 &  \cdots & a_l y_{k} & 0 &  \cdots & 0
\\
* & 0 & \cdots & 0 & b_1 & \cdots & 0 \\
\vdots \\
* & 0 & \cdots & 0 & 0 & \cdots & b_q \\
0 &  1&  \cdots  & 1& 0 & \cdots & 0
\end{array}
\right) (\bz).
}
\end{eqnarray}
Recall here that we have $l >0$. 

We can use the last $q$ columns to cancel the entries marked $``* "$ in the first column without
affecting the remaining entries.
  
  Then, multiplying the first column by $-y_k$ ($\ne 0$ at $\bz$), we obtain
\begin{eqnarray} \nonumber
{\footnotesize
\left(
\begin{array}{cccccccccc}
-a_1x_{(k,s_1)}y_k & a_1y_{k}   & \cdots & 0 & 0  & \cdots &0 \\
\vdots \\
- a_l x_{(k, s_l)} y_k & 0 &  \cdots & a_l y_{k} & 0 &
  \cdots & 0\\
0 & 0 & \cdots & 0 & b_1  & \cdots & 0 \\
\vdots \\
0 & 0 & \cdots & 0 & 0 & \cdots & b_q   \\
0 &  1&  \cdots  & 1& 0 & \cdots & 0
\end{array}
\right) (\bz).
}
\end{eqnarray}
Multiplying the $(i+1)$-th column by $x_{(k, s_i)}$ and adding it to the first column 
for all $1\le i\le  l$, 
we obtain
\begin{eqnarray} \nonumber
{\footnotesize
\left(
\begin{array}{cccccccccc}
0 & a_1y_{k}   & \cdots & 0 & 0  & \cdots &0 \\
\vdots \\
0 & 0 &  \cdots & a_l y_{k} & 0 &  \cdots & 0\\
0 & 0 & \cdots & 0 & b_1   & \cdots & 0 \\
\vdots \\
0 & 0 & \cdots & 0 & 0 & \cdots & b_q  \\
\sum_{i=1}^l  x_{(k, s_i)}  &  1&  \cdots  & 1 & 0 & \cdots & 0
\end{array}
\right) (\bz).
}
\end{eqnarray}
But, at the point $\bz$, by \eqref{finaleq-alpha}, we have
$$\sum_{i=1}^l   x_{(k, s_i)} (\bz) = -  1 \ne 0.$$
Thus, we conclude that 
$J^*(\cB^\gov_{\fV,F_k}, L_{\fV,F_k})$ 
is invertible at $\bz$,
and one sees that all the variables used to compute it are pleasant. 
\end{proof}

 \subsection{Jacobian matrix: Case $(\beta)$}
\label{beta} $\ $

In this case, $x_{(k,0)}(\bz_0) = 0$,
we can assume that $\fV$ lies over $(x_{(k,s_0)} \ne 0)$
for some $s_0 \in [\tau_k]$,
thereby $x_{\fV, (k,s_0)}=1$.
By Theorem \ref{final-forms} (1) and (2) ($\beta$), we obtain
\begin{eqnarray}\label{finaleq-beta}  
B_{\fV, s_0}: \;\; a_0 y_{\fV,k} - c_0        \;\;\;\;\;\;\;\;\;\;\;\; \;\;\;\;\;\;\;\;\;\;\;\;\;          \\ 
  B_{\fV, s_i}: \;\; a_i x_{\fV,(k, s_i)} - c_i, \;\; i \in [l]  \;\;\;\;\;\;\;\;\; \nonumber \\
 B_{\fV, t_i}: \;\;  b_i x_{\fV,(k, t_i)} - d_i , \;\; i \in [q]. \;\;\;\;\;\;\;\;  \nonumber \\
L_{\fV, F_k}: \;\; 1 +  y_{\fV, (k,0)}  \;\;\;\;\;\;\;\;\;\;\;\;\;\;\;\;\;\;\;\;\;\;\;\;\;\; \nonumber 
\end{eqnarray}

Now, we  introduce the following maximal minor of the Jacobian $J(\fG_{\fV,F_k})$                     
$$J^*(\fG_{\fV,F_k})= {{\partial(B_{\fV, s_0},
B_{\fV, s_1} \cdots B_{\fV, s_{\l}}, B_{\fV, t_1} \cdots 
B_{\fV, t_q}, L_{\fV, F_k})}
 \over {{\partial(  y_{ k},
x_{(k, s_1)} \cdots x_{(k, s_l)}}},
x_{(k, t_1)} \cdots x_{(k, t_q)},
 y_{\fV, (k, 0)})}. $$

 \begin{lemma}\label{jjj-beta}
The matrix $J^*(\fG_{\fV,F_k})$ is invertible
at the point $\bz$, and all the variables that are used to compute it are pleasant.
\end{lemma}
\begin{proof}
Then, one calculates and finds that at the point $\bz$, it is equal to
\begin{eqnarray} \nonumber
\left(
\begin{array}{cccccccccc}
 a_0 & 0 & \cdots   & 0  & 0   & \cdots & 0 & 0\\
*  &  a_1  & \cdots &   0  & 0 & \cdots & 0 & 0\\
\vdots \\
*  & 0 & \cdots & a_l   & 0   & \cdots & 0 & 0       \\
*  & 0&  \cdots & 0 & b_1 & \cdots  &0     & 0\\
\vdots \\
*  &  0 & \cdots & 0 & 0 & \cdots & b_l         & 0     \\
 0 & * & \cdots & *&  * & \cdots  & * &    1
\end{array}
\right) (\bz).
\end{eqnarray}
Thus, we conclude that 
$J^*(\fG_{\fV,F_k})$ 
is invertible at $\bz$, and all the variables that are used to compute it are pleasant.
 \end{proof}
 
\subsection{Conclusion}

\begin{thm}\label{tsV=smooth} 
The final scheme $\tsV_\ell$ is smooth.
In particular, on any chart $\fV$, $\tsV_\ell \cap \fV$ 
is defined by $\fG_\fV$.
\end{thm}
\begin{proof} 
We follow the earlier notation.

At any point $\bz \in \fV$, 
we  obtain the  following minor of  $\Jac (\fG_\fV)$
 \begin{equation}\label{the-grand-matrix}  
J^*(\fG_\fV)=\left(
\begin{array}{cccccccccc}
 J^*(\fG_{\fV,F_1})  & 0 & 0& \cdots & 0  \\
 * & J^*(\fG_{\fV,F_2}) &  0 & \cdots & 0  \\
\vdots &    \\
 * &  * & * & \cdots & J^*(\fG_{\fV, F_\up}) \\
 \end{array}
\right)
\end{equation}
where
all the blocks along diagonal are invertible at $\bz$.
Observe that 
$$\rk (J^*(\fG_{\fV}))=\sum_{k=1}^\up
\rk (J^*(\fG_{\fV,F_k}) = \up+ \sum_{k=1}^\up \tau_k.$$
Then,
$$\dim_\bz T_{\tsV_\ell} \le \dim \tsR - \rk J^*(\fG_\fV) $$
$$= \dim \bU + \sum_{k=1}^\up \tau_k -
(\up+ \sum_{k=1}^\up \tau_k)$$
$$= \dim \bU - \up = \dim \bV =\dim \tsV.$$
Hence, $\tsV_\ell$ is smooth, and consequently,
over any chart $\fV$, it is  defined  by 
$$L_{\fV,\sF}, \; \cB_\fV^\gov,$$  implying that all the relations in $\cB^\ngv$ are dependent.
\end{proof}

As all universal blowups are along codimention two loci, using the same calculations, it is routine to show that
any singular stratum $Z_\Ga$ admits a smooth alteration.

{\scriptsize
\tableofcontents
}

\end{document}